\documentclass{amsart}

\usepackage{amsmath}
\usepackage{amssymb}
\usepackage{mathrsfs}  
\usepackage{hyperref}
\usepackage{enumitem}
\usepackage[style=numeric, doi=false, url=false, isbn=false, maxbibnames=99]{biblatex}
\DeclareFieldFormat[article]{title}{#1}
\DeclareFieldFormat[inbook]{title}{#1}
\DeclareFieldFormat[article,periodical]{volume}{\mkbibbold{#1}}

\newtheorem{theorem}{Theorem}[section]
\newtheorem{lemma}[theorem]{Lemma}

\newtheorem{proposition}[theorem]{Proposition}

\theoremstyle{definition}
\newtheorem{definition}[theorem]{Definition}

\newtheorem{remark}[theorem]{Remark}

\let\Im=\undefined\DeclareMathOperator{\Im}{Im}

\DeclareMathOperator{\sgn}{sgn}

\numberwithin{equation}{section}

\allowdisplaybreaks

\begin{document}

\title[dispersive equations with quasi-periodic initial data]{Local wellposedness of dispersive equations with quasi-periodic initial data}

\author[H.~Papenburg]{Hagen Papenburg}
\address{Department of Mathematics, University of California, Los Angeles, CA 90095, USA}
\email{papenburg@math.ucla.edu}

\begin{abstract}
We prove unconditional local well-posedness in a space of quasi-periodic functions for dispersive equations of the form $$\partial_tu + Lu + \partial_x(u^{p+1})=0,$$ where $L$ is a multiplier operator with purely imaginary symbol which grows at most exponentially. The class of equations to which our method applies includes the generalized Korteweg-de Vries equation, the generalized Benjamin-Ono equation, and the derivative nonlinear Schr\"odinger equation. We also discuss well-posedness of some dispersive models which do not have a problematic derivative in the nonlinearity, namely, the nonlinear Schr\"odinger equation and the generalized Benjamin-Bona-Mahony equation, with quasi-periodic initial data. In this way, we recover and improve upon results from \cite{damanik2015existence}, \cite{damanik2021local} and \cite{damanik2022quasiperiodic} by shorter arguments.

\end{abstract}
\maketitle

 \section{Introduction}
  
We investigate well-posedness of some well-studied dispersive equations such as the Korteweg-de Vries equation, the Benjamin-Ono equation, and the derivative nonlinear Schr\"odinger equation for quasi-periodic data.
A function $u:\mathbb{R} \rightarrow \mathbb{C}$ is called quasi-periodic if it takes the form
\begin{equation}\label{quasi-peridic}
u(x)=\sum_{n \in \mathbb{Z}^{\nu}}\widehat{u}(n)e^{i\langle \omega,n \rangle x}, \end{equation} where $\omega=(\omega_1,...,\omega_{\nu})$ is a frequency vector whose entries are linearly independent over the integers. The Fourier coefficients in a quasi-periodic series expansion are uniquely determined by the function; they can be recovered by 
  \begin{equation}\label{coefficientformula}
  \widehat{u}(n) = \lim_{N \rightarrow \infty} \frac{1}{2N} \int_{-N}^{N}u(x)e^{-i\langle \omega, n \rangle x}dx.	
  \end{equation}
  The sum of two periodic functions with incommensurable periods is not periodic, but it is quasi-periodic.
  
 Our main result is a local well-posedness theorem in a quasi-periodic function space for a wide family of dispersive equations. The arguments presented allow us to treat uniformly equations of the form \begin{equation}\label{multiplierequation}
 u_t + Lu + \partial_x(u^{p+1}) = 0	
 \end{equation}
 where $L$ is a Fourier multiplier operator with purely imaginary symbol satisfying a certain exponential growth bound, see \eqref{growthcondition}. Here $u$ is a function of space and time $(t,x) \in \mathbb{R} \times \mathbb{R}$ and $p$ is a positive integer. The generalized Korteweg-de Vries equation  \begin{equation}\tag{g-KdV}\label{g-KdV}\partial_t u + \partial_x^3 u + \partial_x(u^{p+1}) = 0 \end{equation} 
 and the generalized Benjamin-Ono equation 
 \begin{equation}\tag{g-BO}\label{gBO} \partial_tu+H\partial_x^2u+\partial_x(u^{p+1})=0 \end{equation}
 belong to this family of equations. In the g-BO equation, $H$ denotes the Hilbert transform --- the multiplier operator with symbol $-i\sgn(\xi)$, where  $\xi$ denotes the frequency variable. For \eqref{gBO} the operator $L=H\partial_x^2$ has symbol $m(\xi)=-i|\xi|\xi$, for \eqref{g-KdV} the operator $L=\partial_x^3$ has symbol $m(\xi)=-i\xi^3$. 

 The case $p=1$ in (g-KdV) is the classical KdV-equation named after Korteweg and de Vries, who derived this equation in 1895 as a model for the propagation of long waves in shallow water. They were motivated by the the question of whether or not wavefronts must necessarily steepen and sought to explain the observation of waves that keep their shape (now called solitons \cite{ApplicationsofKdV, KortewegdeVries}). The case $p=2$ is known as the modified KdV equation, which is related to KdV via the Miura map. The KdV equation has long attracted attention from mathematicians and physicists due to several remarkable features; in modern terminology, we have an infinite-dimensional, completely integrable Hamiltonian system \cite{Deift2019FiftyYO, PhysRevLett.19.1095, klein2022nonlinear, Lax1968IntegralsON, MiuraGardnerKruskal}. Well-posedness in Sobolev spaces on the real line and the torus  has been extensively studied; see, for example, 
\cite{bourgainglobal, MR1969209, MR3559154, KenigPonceVega2, killip2019kdv, klein2022nonlinear, MR2233925}. For results about (g-)KdV with initial data in spaces of analytic functions we refer to \cite{MR2172859, 10.57262/die/1356060724, inbook}. 
 
In this paper, we consider quasi-periodic initial data whose coefficients $\widehat{u}: \mathbb{Z}^{\nu} \rightarrow \mathbb{C}$ satisfy the decay condition 
\begin{equation}\label{decay-condition}
 \|u\|_{\omega,k}:=\|\widehat{u}(n)e^{k|n|}\|_{l^1(\mathbb{Z}^{\nu})}=\sum_{n \in \mathbb{Z}^{\nu}}|\widehat{u}(n)|e^{k|n|} < \infty
\end{equation} for fixed $k>0$. We define $V^{\omega,k}$ to be the space of all $\omega$-quasi-periodic functions of the form \eqref{quasi-peridic} for which the norm $\|\cdot\|_{\omega,k}$ is finite:
  
  $$V^{\omega,k} = \Big\{ u(x)=\sum_{n \in \mathbb{Z}^{\nu}}\widehat{u}(n)e^{i \langle \omega, n \rangle x} \ \Big| \ \|u\|_{\omega,k} < \infty \Big\}.$$
 Observe that functions in $V^{\omega,k}$ are analytic --- in fact, they admit a holomorphic extension to the strip $\{z \in \mathbb{C}:|\Im(z)|<|\omega|^{-1}k\}$. For $T>0$ we denote by $C([-T,T],V^{\omega,k})$ the space of continuous functions $u:[-T,T] \rightarrow V^{\omega,k}$ equipped with the norm $$||u||_{C_tV_x^{\omega,k}}=\sup_{t \in [-T,T]} ||u(x,t)||_{V_x^{\omega,k}}.$$ We can express a function $u \in C([-T,T],V^{\omega,k})$ and also more general spatially $\omega$-quasi-periodic functions as \begin{equation}\label{seriesexpansion}u(t,x)=\sum_{n \in \mathbb{Z}^{\nu}}\widehat{u}(t,n)e^{i\langle \omega, n \rangle x}\end{equation} with coefficients $\widehat{u}(t,n)$ given by \eqref{coefficientformula} for fixed time $t$, and will do so throughout.
  We denote by $V^{\omega,k}_{\mathbb{R}}$ the subspace of all real-valued functions in $V^{\omega,k}$, which amounts to imposing the additional condition $\widehat{u}(-n)=\overline{\widehat{u}(n)}$ on the coefficients. The space $C([-T,T],V_{\mathbb{R}}^{\omega,k})$ is defined accordingly. For a quasi-periodic function of the form \eqref{quasi-peridic} the multiplier $L$ with symbol $m$ is defined by \begin{equation}
      Lu(x) = \sum_{n \in \mathbb{Z}^{\nu}} m(\langle \omega, n \rangle)\widehat{u}(n)e^{i\langle \omega, n \rangle x}.
  \end{equation}
  
In addition to being purely imaginary, we require the symbol of the multiplier operator $L$ to satisfy the exponential growth bound
 \begin{equation}\label{growthcondition}
 	\sup_{\xi}|m(\xi)|e^{-\frac{k - \varepsilon}{|\omega|}|\xi|} < \infty 
 \end{equation} for some $\epsilon>0$. This ensures that $L$ maps $V^{\omega,k}$ into $V^{\omega,\varepsilon}$, guaranteeing that $Lu$ is smooth in the spatial variable. Requiring that \begin{equation}\label{symmetrycondition} m(-\xi)=\overline{m(\xi)}\end{equation} ensures that $L$ maps real-valued functions to real-valued functions. We say that $u \in C([-T,T],V^{\omega,k})$ [or $u \in C([-T,T],V_{\mathbb{R}}^{\omega,k})$] is a classical solution to the equation \eqref{multiplierequation} if and only if $u$ is at least once continuously differentiable in the time variable and solves the equation pointwise. We are now ready to state the following result:
 
   \begin{theorem}\label{multipliertheoremreal}
  	Let $L$ be a Fourier multiplier operator with purely imaginary symbol satisfying the growth bound \eqref{growthcondition} and the symmetry condition \eqref{symmetrycondition}, and let $k>\kappa>0$. For initial data in $V_{\mathbb{R}}^{\omega,k}$, the Cauchy problem
  \begin{equation}\label{multiplierCauchyproblem}
  	\partial_t u + Lu + \partial_x(u^{p+1}) = 0, \quad u(0,x)=u_0(x)
  \end{equation}
  is unconditionally locally well-posed in $V_{\mathbb{R}}^{\omega,\kappa}$ both forward and backward in time.
  \end{theorem}
  
  The notion of well-posedness here and throughout the paper is as follows: for any $u_0 \in V^{\omega,k}_{\mathbb{R}}$ there exists an open ball $B \subseteq V^{\omega,k}_{\mathbb{R}}$ containing $u_0$ and a $T>0$ such that for any initial data $\tilde{u}_0 \in B$ there exists a unique classical solution in $C([-T,T], V^{\omega,\kappa}_{\mathbb{R}})$ and the data-to-solution map $B \rightarrow C([-T,T],V^{\omega,\kappa}_{\mathbb{R}})$ is Lipschitz continuous. The time of existence can in fact be taken to depend only on the size of the initial data, that is, there exists a time $T:=T(R)$ that works uniformly for all initial data with $\|u_0\|_{V^{\omega,k}_{\mathbb{R}}} < R$.   
  
  The \eqref{g-KdV} and \eqref{gBO} equations, which belong to the class of equations to which this theorem applies, are traditionally studied for real-valued functions. That is the reason why we stated the above theorem using the space $V_{\mathbb{R}}^{\omega,\kappa}$ of real-valued functions in $V^{\omega,\kappa}$. However, the restriction to real-valued functions is insubstantial in the proof, and we could formulate an analogous well-posedness result in the space $V^{\omega, \kappa}$. In that setting the symmetry condition \eqref{symmetrycondition} on the symbol can be dropped. We may also replace the nonlinearity $\partial_x(u^{p+1})$ by $\partial_x(P(u,\bar{u}))$ with $P$ being a polynomial of two variables and formulate the following well-posedness result, which applies in particular to the derivative nonlinear Schr\"odinger equation
 \begin{equation}\tag{d-NLS}\label{dNLS}
 \partial_t u + i\partial_x^2u +\partial_x(|u|^{2}u)=0.	
 \end{equation}
  
  \begin{theorem}\label{multipliertheorem}
  	Let $L$ be a multiplier operator with purely imaginary symbol satisfying the growth bound \eqref{growthcondition}, let $P$ be a polynomial of two variables, and let $k>\kappa>0$. For initial data in $V^{\omega,k}$, the Cauchy problem
  \begin{equation}\label{multiplierCauchyproblem2}
  	\partial_t u + Lu + \partial_x(P(u,\bar{u})) = 0, \quad u(0,x)=u_0(x)
  \end{equation}
  is unconditionally locally well-posed in $V^{\omega,\kappa}$ both forward and backward in time. \end{theorem}

 In \cite{damanik2015existence}, the Cauchy problem for KdV was studied with quasi-periodic initial data of the form (\ref{quasi-peridic}) and coefficients satisfying an exponential decay condition \begin{equation}\label{decay}|\widehat{u}_0(n)| \lesssim e^{-k|n|}.\end{equation} 
  It is shown that locally in time, there exists a spatially quasi-periodic classical solution of the form \eqref{seriesexpansion} whose Fourier coefficients satisfy a decreased exponential decay bound \begin{equation}\label{decreaseddecay}|\widehat{u}(t,n)| \lesssim e^{-\frac{k}{2}|n|}\end{equation} uniformly for $t \in [0,T]$ for some $T>0$. 
  Further, we have uniqueness of solutions of the form  \eqref{seriesexpansion} satisfying an exponential decay estimate on coefficients $|\widehat{u}(t,n)| \lesssim e^{- \rho|n|}$ uniformly in $t \in [0,T]$ for some $\rho>0$ and small enough $T>0$ (see \cite[Theorem A]{damanik2015existence}). 
  
  The approach taken in \cite{damanik2015existence} is to examine the associated integral equation for the Fourier coefficients (see \eqref{integralequation}) and obtain estimates on the iterative application of an integral transformation. This allows the authors to execute a Picard iteration scheme and construct a solution. 
  In \cite{damanik2021local}, this approach is taken to obtain an analogous local existence and uniqueness result for the \textit{generalized} KdV equation with quasi-periodic initial data with exponentially decaying Fourier coefficients as in \eqref{decay}; see \cite[Theorem 0.1 and 0.2]{damanik2021local}. These existence and uniqueness results in \cite{damanik2015existence, damanik2021local} can be qualitatively recovered as a corollary to our well-posedness theorem, thereby providing a shorter proof. The details are provided at the end of Section~\ref{preliminaries}, where the slight differences between the decay conditions \eqref{decay}, \eqref{decreaseddecay} and the spaces $V^{\omega,k}$, $C([0,T],V^{\omega,\kappa})$ are addressed. In addition, our theorem adds a statement about continuous dependence of solutions on initial data.

  The approach to the Cauchy problem \eqref{multiplierCauchyproblem} in this paper also relies on the associated Duhamel's formula
   \begin{equation}\label{Duhamelsformula}
  u(t)=e^{-tL}u_0 + \int_0^t e^{-(t - \tau)L}\partial_x(u^{p+1}(\tau)) d\tau	.
  \end{equation}
  
  Here $e^{-tL}$ is the propagator of the corresponding linear evolution equation $\partial_tu=-Lu$, that is, the multiplier operator defined by $$\widehat{e^{-tL}v}(\xi)=e^{-tm(\xi)}\widehat{v}(\xi).$$
As the symbol $m(\xi)$ is purely imaginary, $t\mapsto e^{-tL}$ defines a group action of $\mathbb{R}$ on $V^{\omega,\kappa}$ by isometry for any $\kappa>0$.   
For a function $u \in C([-T,T],V^{\omega,\kappa})$ we will observe that the integrand $e^{-(t-\tau)L}\partial_x (u^{p+1}(\tau))$ is in $C([0,t],V^{\kappa-\delta})$ for any positive $\delta<k$ and thus uniformly bounded in both time and space, so that the RHS in \eqref{Duhamelsformula} can be interpreted pointwise. 
We will characterize classical solutions to the Cauchy problem \eqref{multiplierCauchyproblem} by Duhamel's formula \eqref{Duhamelsformula}, see Proposition \ref{notionofsolution}.

 We define the transformation
 \begin{equation}\label{Duhameloperator}
  (\Gamma_{u_0}u)(t)=e^{-tL}u_0 + \int_0^t e^{-(t - \tau)L}\partial_x (u^{p+1}(\tau)) d\tau	.
  \end{equation}
This will be shown to be a contraction on suitable subsets of the Banach spaces $X^{k,\gamma,T}$ introduced in Definition~\ref{kgammaTspaces}. 
The existence of solutions follows from Banach's fixed point theorem and the inclusion $X^{k,\gamma,T}\subseteq C([-T,T],V^{\omega,\kappa})$. Using the inclusion $C([-T,T],V^{\omega,\kappa}) \subseteq X^{\kappa-\epsilon, \gamma,T}$ and a bootstrap argument we will get unconditional uniqueness in $C([-T,T],V^{\omega,\kappa})$. From the contraction property we also get continuous dependence on initial data in the sense that the data-to-solution map $B_{V^{\omega,k}}(0,R) \rightarrow C([-T,T],V^{\omega,\kappa})$ is Lipschitz.
The key of the proof is in establishing estimates for the operator \eqref{Duhameloperator} in order to show its contractive property; these are presented in Propositions \ref{linear1}, \ref{multilinear} below. 

We remark that our approach is very robust --- neither does it utilize the dispersive nature of the equation nor does it impose any Diophantine condition on the frequency vector $\omega$. By exploiting dispersion and using the Fourier restriction norm method, Tsugawa showed well-posedness of the KdV equation (\eqref{g-KdV} with $p=1$) in a Sobolev-type space of quasi-periodic functions defined by the norm \begin{equation}\label{tsugawanorm}\|u\|_{G^{\sigma,-1/2}}= \Big\| \ |\langle \omega, n \rangle|^{-\frac{1}{2}} \cdot \prod_{i=1}^{\nu} \langle n_i \rangle^{\sigma_i} \cdot \widehat{u}(n) \Big\|_{l^2(\mathbb{Z}^{\nu})},\end{equation}
where $\sigma=(\sigma_1,...,\sigma_{\nu}) \in \mathbb{R}_{\geq 0}^{\nu}$ are parameters satisfying some constraints, see \cite[Theorem 1.1, Proposition 4.1]{tsugawa2012local}. The norm \eqref{tsugawanorm} heavily penalizes Fourier coefficients $\widehat{u}(n)$ for which the frequency  $ \langle \omega, n \rangle $ is small in absolute value. 

The following example illustrates that for some frequency vectors $\omega$, this can lead the above norm \eqref{tsugawanorm} to diverge even if the Fourier coefficients decay exponentially, that is, even for analytic $u$. Take 
$$\omega=(1,\alpha) \quad \text{with} \quad \alpha = \sum_{n=1}^{\infty} \frac{1}{10 \uparrow \uparrow n},$$ where Knuth's up-arrow notation $10 \uparrow \uparrow n$ denotes $n$-fold repeated exponentiation $10^{10^{10^{\cdots}}}$. For $n\in\mathbb{N}$ take integers $q_n:=10\uparrow \uparrow n$ and $p_n:=\sum_{m\leq n} \frac{10 \uparrow \uparrow n}{10 \uparrow \uparrow m}$. Note that $p_n < q_n$. Then 
\begin{equation*}
    |\langle \omega, (-p_n,q_n) \rangle| = q_n\Big| \alpha - \frac{p_n}{q_n} \Big| = \sum_{m>n} \frac{10 \uparrow \uparrow n}{10 \uparrow \uparrow m} =: C(n),
\end{equation*}
Observing that $C(n) \leq q_n^{-n+1}$ amounts to saying that $\alpha$ is a Liouville number. There is plenty of room to spare in this inequality and we can bound 
$$C(n) \leq 2 \cdot \frac{q_n}{10^{q_n}} \lesssim 9^{-q_n} \leq e^{-4\kappa q_n} \leq e^{-2 \kappa|(-p_n,q_n)|}$$ provided that $\kappa \leq \frac{\ln{9}}{4}$. Thus, $|\langle (1, \alpha),  (-p_n, q_n) \rangle |^{-1/2} = C(n)^{-\frac{1}{2}} \gtrsim e^{\kappa |(-p_n,q_n)|}$ for infinitely many $(-p_n,q_n)$. Therefore, the analytic quasi-periodic function $u(x)=\sum_{n \in \mathbb{Z}^{2}}e^{-\kappa|n|}e^{i \langle \omega, n \rangle x}$ does not belong to the space $G^{\sigma,-1/2}$.

What brings together the equations \eqref{g-KdV}, \eqref{gBO} and \eqref{dNLS} subsumed by Theorems \ref{multipliertheoremreal} \& \ref{multipliertheorem} is the derivative in the nonlinearity, which is the cause of the main difficulty in studying the associated Cauchy problem.
In the absence of a derivative in the nonlinearity, as for the nonlinear Schr\"odinger equation, a simpler argument can be employed. This simpler argument also applies to the generalized Benjamin-Bona-Mahony equation, as we will see in Section~\ref{noderivativemodels}.

 The paper is organized as follows. In Section \ref{preliminaries}, we present preliminaries needed in the proof of Theorems \ref{multipliertheoremreal} and \ref{multipliertheorem}, most notably estimates for the operator \eqref{Duhameloperator}. In Section \ref{Proof}, we put everything together and present the proof of the theorems just mentioned. In Section \ref{noderivativemodels}, unconditional local well-posedness of the nonlinear Schr\"odinger equation and the generalized Benjamin-Bona-Mahony equation in some spaces of quasi-periodic functions is shown.

  \subsection*{Acknowledgements}
  The author was supported by NSF grant DMS-2154022 and NSF grant DMS-2054194. The author would like to thank Rowan Killip and Monica Vi\c{s}an for helpful discussions and for providing feedback on earlier versions of the paper.

  \section{Preliminaries}\label{preliminaries}
We begin by showing that the notion of classical solution to the Cauchy problem \eqref{multiplierCauchyproblem} and Duhamel's formula are equivalent for functions in $C([-T,T],V^{\omega,\kappa})$.

 \begin{proposition}\label{notionofsolution}
  	Let $u \in C([-T,T],V^{\omega,\kappa})$. The following are equivalent:
  	\begin{enumerate}
  	\item $u$ is a classical solution to the Cauchy problem \eqref{multiplierCauchyproblem} with initial data $$u_0(x)=\sum_{n}\widehat{u}_0(n)e^{i\langle \omega, n \rangle x}.$$	
  	\item The Fourier coefficients given by \eqref{seriesexpansion} 
  	satisfy the following system of coupled integral equation:
  \begin{multline}\label{integralequation}
  	\widehat{u}(t,n)=\widehat{u}_0(n)e^{-tm(\langle \omega, n \rangle)} \\ -i\langle \omega, n\rangle \int_0^t\Big[e^{-m(\langle \omega, n \rangle)(t - \tau)}\sum_{q_1+\cdots+q_{p+1}=n} \widehat{u}(\tau,q_1)\cdot\cdot \cdot \widehat{u}(\tau,q_{p+1}) \Big]d \tau.
  	\end{multline}
  	\item Duhamel's formula \eqref{Duhamelsformula} holds pointwise throughout $[-T,T] \times \mathbb{R}$.
  	\end{enumerate}  
\end{proposition} 
  
The following fact is used several times in the proof.
  \begin{proposition}\label{BanachalgebraVk}
  	The space $V^{\omega,k}$ is a commutative  Banach algebra.
  \end{proposition}
  
  \begin{proof} 
  	For $u,v \in V^{\omega,k}$, we may bound
  	\begin{multline}
  		\|uv\|_{\omega,k}=\sum_n |\widehat{uv}(n)|e^{k|n|} =\sum_n \Big|\sum_{q+r=n} \widehat{u}(q)\widehat{v}(r)\Big|e^{k|n|} \\ \leq \sum_q |\widehat{u}(q)|e^{k|q|} \cdot \sum_r |\widehat{u}(r)|e^{k|r|} = \|u\|_{\omega,k} \cdot \|v\|_{\omega,k}.
  	\end{multline}
   As the space is defined by a weighted $l^1$-norm, it is complete. 
  \end{proof}

 Note that the coefficients of the $(p+1)$-th power of a function $u \in V^{\omega,k}$ are given by $(p+1)$-fold discrete convolution which we denote by  $$[*^{p+1}\widehat{u}](n):=\sum_{q_1+\cdots + q_{p+1} = n}\widehat{u}(q_1)\cdot \cdot \cdot \widehat{u}(q_{p+1}).$$

  \begin{proof}[Proof of Proposition \ref{notionofsolution}]
 Observe that for a function $u \in C([-T,T],V^{\omega,\kappa})$, we have that $Lu + \partial_x(u^{p+1})$ is in $C([-T,T],V^{\omega,\varepsilon})$, where $\varepsilon$ is as in \eqref{growthcondition}. Indeed, differentiation $\partial_x:V^{\omega,\kappa} \rightarrow V^{\omega, \varepsilon}$ is a bounded linear operator as $$\|\partial_xu\|_{\omega,\varepsilon} = \sum_{n \in \mathbb{Z}^{\nu}}|\langle \omega, n \rangle||\widehat{u}(t,n)|e^{\varepsilon|n|} \leq \Big[|\omega|\sup_{n \in \mathbb{Z}^{\nu}}|n|e^{(\varepsilon -\kappa) |n|}\Big] \cdot \|u\|_{\omega,k} \lesssim \|u\|_{\omega,k}.$$ Together with the Banach algebra property this yields $\partial_x(u^{p+1}) \in C([-T,T],V^{\omega,\varepsilon})$, and the growth bound \eqref{growthcondition} guarantees that $Lu \in C([-T,T],V^{\omega, \varepsilon}).$ Therefore, $Lu+\partial_x(u^{p+1})= \sum b(t,n)e^{i \langle \omega, n \rangle x}$ is a spatially quasi-periodic function whose coefficients $$b(t,n)=m(\langle \omega, n \rangle) \widehat{u}(t,n) + i \langle \omega, n \rangle [*^{p+1}\widehat{u}](t,n)$$ satisfy 
  \begin{equation}\label{coefficientbound}
      |b(t,n)| \lesssim e^{- \varepsilon |n|}
  \end{equation}
  uniformly in time. 
  
  $(1)\implies(2)$:
 Suppose that $u \in C([-T,T],V^{\omega,\kappa})$ is a classical solution. The bound \eqref{coefficientbound} allows us to integrate $$\partial_t u(t,x) = -L u(t,x) - \partial_x(u^{p+1})(t,x) = \sum_{n \in \mathbb{Z}^{\nu}}-b(t,n)e^{i \langle \omega, n \rangle x}$$ in time through the summation, so that 
  $$u(t,x)=\sum_{n \in \mathbb{Z}^{\nu}}\widehat{u}(t,n)e^{i\langle \omega, n \rangle x}=\sum_{n \in \mathbb{Z}^{\nu}}\Big(\widehat{u}_0(n)-\int_0^t b(\tau,n) d\tau \Big)e^{i \langle \omega, n \rangle x}.$$
  Comparing coefficients we see that $\widehat{u}(t,n)$ is continuously differentiable in time with 
  $$\partial_t\widehat{u}(t,n)=-b(t,n)=-m(\langle \omega, n \rangle) \widehat{u}(t,n) - i \langle \omega, n \rangle [*^{p+1}\widehat{u}](t,n).$$
  Applying the integrating factor $e^{tm(\langle \omega, n \rangle)}$ gives the desired integral identity \eqref{integralequation}. 
  
    $(2)\implies(3)$: multiplying the integral equation \eqref{integralequation} by the character $e^{m(\langle \omega, n \rangle)x}$ and summing over $n \in \mathbb{Z}^{\nu}$, we get that
    \begin{align*}\label{integralequation}
  	 u(t,x)=\sum_n \widehat{u}(t,n)e^{i \langle \omega, n \rangle x} &=\sum_n \widehat{u}_0(n)e^{-tm(\langle \omega, n \rangle)}e^{i \langle \omega, n \rangle x} \\ - \sum_n  \int_0^t\Big[e^{-m(\langle \omega, n \rangle)(t - \tau)}&\cdot i\langle \omega, n\rangle\sum_{q_1+\cdots+q_{p+1}=n} \widehat{u}(\tau,q_1)\cdots \widehat{u}(\tau,q_{p+1}) \Big]d \tau \cdot e^{i\langle \omega, n \rangle x} \\
  	&= e^{-tL}u_0 + \int_0^t e^{-(t - \tau)L}\partial_x(u^{p+1}(\tau))d \tau.
    \end{align*}
In the last line we interchanged summation over $n$ and integration in the time variable $\tau$. This is justified by uniform convergence of the series expansion of the integrand on the right-hand side. Indeed, the terms of that series satisfy the bound $$\Big|e^{-m(\langle \omega, n \rangle)(t - \tau)} i\langle \omega, n\rangle \sum_{q_1+\cdots+q_{p+1}=n} \widehat{u}(\tau,q_1)\cdot\cdot \cdot \widehat{u}(\tau,q_{p+1}) \Big| \lesssim_{|\omega|} |n|e^{-k|n|} \lesssim_{\delta} e^{-(k - \delta)|n|}$$ for any $\delta>0$. This bound follows as the integrand $e^{-(t - \tau)L}\partial_x(u^{p+1}(\tau))$ belongs to the space $C([0,t],V^{\omega,\kappa-\delta})$ due to the Banach algebra property, boundedness of differentiation $\partial_x:V^{\omega,\kappa} \rightarrow V^{\omega,\kappa-\delta}$, and the fact that the propagator acts by isometry.

   $(3)\implies (1)$: As noted previously $e^{-(t - \tau)L}\partial_x(u^{p+1}(\tau))$ is in $C([0,t],V^{\kappa-\delta})$ for any $0<\delta<\kappa$ and has quasi-periodic series expansion given by $$e^{-(t - \tau)L} \partial_x u^{p+1} = \sum \Big[ e^{-(t - \tau)m(\langle \omega, n \rangle)}i\langle \omega, n \rangle[*^{p+1}\widehat{u}](\tau,n)\Big] e^{i\langle \omega,n\rangle x}.$$
   Upon plugging this into Duhamel's formula, this allows us to interchange integration and summation (over $n$) so that we obtain 
   \begin{multline*}
  	u(t,x)=\sum_{n \in \mathbb{Z}^{\nu}}\widehat{u}(t,n)e^{i\langle \omega, n \rangle x}  = \sum_{n} \bigg[\widehat{u}_0(n)e^{-tm(\langle \omega, n \rangle)}  \\-i\langle \omega, n\rangle e^{-t m(\langle \omega, n \rangle)} \! \int_0^t \! \Big[e^{\tau m(\langle \omega, n \rangle)} \!\!\!\! \sum_{q_1+\cdots+q_{p+1}=n} \!\!\!\! \widehat{u}(\tau,q_1)\cdots \widehat{u}(\tau,q_{p+1}) \Big]d \tau \bigg]e^{i \langle \omega, n \rangle x}.
  	\end{multline*}
   It follows that $u$ is continuously differentiable in time with derivative given by the spatially quasi-periodic function with Fourier coefficients 
   $$-m(\langle \omega, n \rangle)\widehat{u}_0(n)e^{-tm(\langle \omega, n \rangle)}-i\langle \omega, n \rangle \!\!\!\! \sum_{q_1+\cdots+q_{p+1}=n} \!\!\!\! \widehat{u}(q_1,t)\cdot\cdot \cdot \widehat{u}(q_{p+1},t).$$
   That is, \eqref{multiplierCauchyproblem} is satisfied:
   $\partial_t u = -Lu -\partial_x (u^{p+1})=0$ and $u(0,x)=u_0(x)$.   
  \end{proof}

  \subsection{Linear and multilinear estimates}\label{estimatessection}
  In order to show the contractive property of the transformation \eqref{Duhameloperator} we establish both linear and multilinear estimates.

\begin{definition}\label{kgammaTspaces}
  	For $\gamma,T,k>0$ such that $k-\gamma T >0$, define $X^{k,\gamma,T}$ to be the space of all spatially $\omega$-quasi-periodic functions on $[-T,T]$ with continuous Fourier coefficients $\widehat{u}:\mathbb{Z}^{\nu} \times [-T,T] \rightarrow \mathbb{C}$ for which the following norm is finite: $$||u||_{k,\gamma,T}:=\sum_{n \in \mathbb{Z}^{\nu}} \sup_{t \in [-T,T]}|\widehat{u}(n,t)|e^{(k-\gamma |t|)|n|} \ .$$
  \end{definition}

  Two key features of this norm are the exponential weights $e^{(k-\gamma |t|)|n|}$ with parameter decaying linearly in time, and that the supremum over time is inside the sum over frequencies. Both features are of crucial importance in establishing the contractive property of the mapping \eqref{Duhameloperator}; see Proposition \ref{linear1} and Remark \ref{keyfeatures}. 
 
 \begin{proposition}[Linear estimates]\label{linear1}
 	For any initial data $u_0 \in V^{\omega,k}$
 	\begin{equation}\label{firstlinear}
 	\|e^{-tL}u_0\|_{k,\gamma,T} = \|u_0\|_{\omega,k}.	
 	\end{equation}
 	For $W \in X^{k,\gamma,T}$ we have that 
 	\begin{equation}\label{secondlinear}
 		\Bigl\| \int_0^te^{-(t-\tau)L} \partial_x W(\tau) d\tau \Bigr\|_{k,\gamma,T} \lesssim_{|\omega|} \frac{1}{\gamma}\|W\|_{k,\gamma,T}\ .
 	\end{equation}
 \end{proposition}

 \begin{proof} As the symbol of the multiplier operator $L$ is purely imaginary, \eqref{firstlinear} is clear. Using the triangle inequality we get that
 	\begin{multline*}
 	\Bigl\| \int_0^te^{-(t-\tau)L} \partial_x W(\tau) d\tau \Bigr\|_{k,\gamma,T} \leq \sum_{n \in \mathbb{Z}^{\nu}} \sup_{t \in [-T,T]} \bigg| \int_0^t |\langle \omega, n \rangle||\widehat{W}(\tau,n)| e^{(k-\gamma |t|)|n|} d \tau \bigg|. 
 	\end{multline*}
 	Distributing the exponential weight $e^{(k - \gamma |t|)|n|}=e^{(k - \gamma |\tau|)|n|}e^{-\gamma(|t| - |\tau|)|n|}$ and using
 	\begin{equation}\label{keystep}
 		\bigg| \int_0^t |n|e^{-\gamma(|t| - |\tau|)|n|} d\tau \bigg| \leq \frac{1}{\gamma}
 	\end{equation}
 	we further bound this by 
 	\begin{multline*}
 	\sum_{n \in \mathbb{Z}^{\nu}} \sup_{t \in [-T,T]} |\omega|\bigg| \int_0^t |n|e^{-\gamma(|t| - |\tau|)|n|} d\tau \bigg| \sup_{\tau \in [0,t]} |\widehat{W}(\tau,n)|e^{(k-\gamma |\tau|)|n|} \\
 	\leq \frac{|\omega|}{\gamma} \sum_{n \in \mathbb{Z}^{\nu}} \sup_{\tau \in [-T,T]} |\widehat{W}(\tau,n)|e^{(k-\gamma |\tau|)|n|} = \frac{|\omega|}{\gamma}\|W\|_{k,\gamma,T},
 	\end{multline*} which completes the proof of \eqref{secondlinear}. \qedhere
 \end{proof}
 
 \begin{remark}\label{keyfeatures}
  Observe that the linear decay in time in the exponential weight $e^{(k - \gamma |t|)|n|}$ allows us to handle the linear growth of the integrand in $|n|$ arising from the factor $|\langle \omega, n \rangle|$ by utilizing inequality \eqref{keystep}.
  This linear growth in $|n|$ is how the derivative in the nonlinearity manifests on the Fourier side, and poses the main difficulty for establishing the contractive property. The norm $\|\cdot\|_{k,\gamma,T}$ was precisely designed to overcome this difficulty. 
  
 To illustrate this point, compare with the generalized Benjamin-Bona-Mahony equation discussed in Section \ref{noderivativemodels}. Taking the Fourier transform of the equation \eqref{gBBMODE} associated with \eqref{gBBM}, there is no linear growth in $|n|$ in the differential equation for the Fourier coefficients. In this sense, the derivative in the nonlinearity of \eqref{gBBM} is harmless. In this setting, easier tools are available to establish well-posedness, see Section~\ref{noderivativemodels}. 
 
 Moreover, note that in the proof of the second linear estimate \eqref{secondlinear} it is important that in the definition of the norm $\|\cdot \|_{k,\gamma,T}$ the supremum is inside the sum. 
 \end{remark}

  \begin{proposition} [Multilinear estimate]\label{multilinear}
  	The space $X^{k,\gamma,T}$ is a commutative Banach algebra, invariant under spatial translations. Concretely, for an integer $p\geq 1$ and $u_1,..,u_{p+1} \in X^{k,\gamma,T}$ we have that
  $$\|u_1\cdot \cdot \cdot u_{p+1}\|_{k,\gamma,T} \leq \|u_1\|_{k,\gamma,T} \cdot \cdot \cdot \|u_{p+1}\|_{k,\gamma,T}.$$
  \end{proposition}

  \begin{proof}
  The proof of Proposition \ref{BanachalgebraVk} adapts to this setting; 
  for $u,v \in X^{k,\gamma,T}$ we may bound \begin{equation*}\begin{split}
  	\|u \cdot v \|_{k,\gamma,T} & = \sum_{n \in \mathbb{Z}^{\nu}} \sup_{t \in [-T,T]} |\widehat{uv}(t,n)|e^{(k - \gamma |t|)|n|} \\ &  =  \sum_{n \in \mathbb{Z}^{\nu}} \sup_{t \in [-T,T]} \left| \sum_{q+r=n}\widehat{u}(t,q)\widehat{v}(t,r)\right|e^{(k - \gamma |t|)|n|} \\ & \leq \bigg[ \sum_{q} \sup_{t\in[-T,T]}|\widehat{u}(t,q)|e^{(k - \gamma |t|)|q|} \bigg] \cdot \bigg[ \sum_{r} \sup_{t\in[-T,T]}|\widehat{u}(t,r)|e^{(k - \gamma |t|)|r|} \bigg] \\ & = \|u\|_{k,\gamma,T} \cdot \|v\|_{k,\gamma,T}.
   \end{split}
  \end{equation*}
  Completeness of the space $X^{k,\gamma,T}$ follows from the standard proof of completeness of $l^1(\mathbb{Z}^{\nu})$ combined with completeness of $C([-T,T],\mathbb{C})$. Spatial translation amounts to altering the Fourier coefficients by a complex phase, thereby leaving the norm unchanged. \qedhere
  \end{proof}
 Finally, we relate the spaces $X^{k,\gamma,T}$ to the space $C([-T,T],V^{\omega,\kappa})$. Note that for $k-\gamma T> \kappa > 0$ we have a norm-decreasing inclusion $X^{k,\gamma,T} \subseteq C([0,T],V^{\omega,\kappa})$. In the proof of uniqueness in Theorem \ref{multipliertheoremreal}, we will also use the following embedding of normed spaces.
 \begin{lemma}\label{embedding}
  	For any $\kappa> \varepsilon>0$, and $\gamma,T>0$ with $\kappa-\varepsilon-\gamma T>0$ we have $$C([-T,T],V^{\omega, \kappa}) \subseteq X^{\kappa -\varepsilon, \gamma, T}.$$ In fact, $$\|u\|_{\kappa-\varepsilon,\gamma,T} \leq Q\|u\|_{C_tV_x^{\omega,\kappa}}$$ with $Q=\sum_{n \in \mathbb{Z}^{\nu}} e^{- \varepsilon|n|}$.
  \end{lemma}
  
  \begin{proof}
  For $$u(t,x) = \sum_{n \in \mathbb{Z}^{\nu}} \widehat{u}(t,n)e^{i \langle \omega, n \rangle x} \in C([-T,T],V^{\omega,\kappa}),$$ we have $|\widehat{u}(t,n)| \leq \|u\|_{C_t V^{\omega,\kappa}_x}e^{-\kappa|n|}$ uniformly in $n$ and $t$, and therefore \begin{equation*}\|u\|_{\kappa - \varepsilon, \gamma,T} \lesssim \|u\|_{C_tV_x^{\omega,\kappa}} \cdot \sum_{n \in \mathbb{Z}^{\nu}}e^{-\varepsilon |n|} \lesssim \|u\|_{C_tV_x^{\omega,\kappa}}. \qedhere \end{equation*} 
  \end{proof}

 \section{Proof of Theorems \ref{multipliertheoremreal} and \ref{multipliertheorem}}\label{Proof}
 We first present the proof of Theorem \ref{multipliertheoremreal}. Afterwards, we discuss the minor adjustments required to prove Theorem \ref{multipliertheorem} regarding \eqref{dNLS}.

 \begin{proof}[Proof of Theorem \ref{multipliertheorem}]
 	The following inequality establishes local Lipschitz continuity of the transformation \eqref{Duhameloperator}. Given any initial data $u_0 \in V^{\omega,k}$ we may combine the second linear estimate (Proposition \ref{linear1}) and the multilinear estimate (Proposition \ref{multilinear}) to bound 
\begin{equation}
 \begin{split}\| \Gamma_{u_0}(u)-\Gamma_{u_0}(v) \|_{k,\gamma,T} & = \Bigl\| \int_0^te^{-(t-\tau)L} \partial_x (u^{p+1}-v^{p+1})(s) ds \Bigr\|_{k,\gamma,T}  \\ 
 		& \lesssim_{|\omega|} \frac{1}{\gamma}\|u^{p+1} - v^{p+1} \|_{k,\gamma,T} \\ & \leq \frac{1}{\gamma} [\|u\|_{k,\gamma,T}+\|v\|_{k,\gamma,T}]^p \cdot  \|u-v\|_{k,\gamma,T}.
   \end{split}
 	\end{equation}
Arguing similarly, using Proposition~\ref{linear1} and Proposition~\ref{multilinear}, we get that
 	\begin{equation}
 	\|\Gamma_{u_0}u\|_{k,\gamma,T} \leq \|u_0\|_{\omega,k} + \frac{|\omega|}{\gamma}\|u\|_{k,\gamma,T}^{p+1}.	
 	\end{equation}

 	Given $R>0$, we can take $\gamma>0$ large enough and $T>0$ small enough so that for all initial data $\tilde{u}_0 \in V^{\omega,k}$ of size $\|u_0\|_{\omega,k}<R$, the mapping $\Gamma_{\tilde{u}_0}$ is a contraction on $B_{X^{k,\gamma,T}}(0,2R)$.
 Taking $T>0$ small enough so that $k -\gamma T > \kappa$ guarantees the inclusion $X^{k,\gamma,T} \subseteq C([-T,T],V^{\omega,\kappa})$. This gives existence of solutions with uniform time horizon $T$ for initial data in $B_{V^{\omega,k}}(0,R)$ by Banach's fixed point theorem applied to $\Gamma_{\tilde{u}_0}$. The embedding $X^{k,\gamma,T} \subseteq C([-T,T],V^{\omega,\kappa})$ ensures that this is in fact a solution in $C([-T,T],V^{\omega,\kappa})$. 
 	 
Next, we show unconditional uniqueness in the space $C([-T,T],V^{\omega,\kappa})$ for suitable $T>0$. 
We choose $\gamma>0$ large enough and $T>0$ small enough so that the following hold simultaneously:
\begin{enumerate}
    \item  For all initial data $\tilde{u}_0 \in V^{\omega,k}$ of size $\|u_0\|_{\omega,k}<R$ the associated mapping $\Gamma_{\tilde{u}_0}$ is a contraction on $B_{X^{k,\gamma,T}}(0,2R)$.
  	\item For all initial data $\tilde{u}_0 \in V^{\omega,k}$ of size $\|u_0\|_{\omega,k}<R$, the associated mapping $\Gamma_{\tilde{u}_0}$ is a contraction on both $B_{X^{\kappa-\varepsilon,\gamma,T}}(0,2RQ)$ and $B_{X^{\kappa-\varepsilon, \gamma, T}}(0,3RQ)$, where $Q=\sum_{n \in \mathbb{Z}^{\nu}} e^{-\varepsilon|n|}$ is as in Lemma~\ref{embedding}.
 \end{enumerate} 
Indeed, we may choose $$\gamma \gtrsim_{|\omega|} \max(R^p,(QR)^p) =(QR)^p,$$ and then take $T>0$ small enough such that $k-\varepsilon - \gamma T > \kappa$, thereby guaranteeing the embedding $C([-T,T],V^{\omega, \kappa}) \hookrightarrow X^{\kappa -\varepsilon, \gamma, T}$ and the norm decreasing inclusion $X^{k,\gamma,T} \subseteq C([-T,T],V^{\omega,\kappa})$. 

 Condition (1) guarantees existence of solutions with uniform time horizon $T>0$ for all initial data in $B_{V^{\omega,k}}(0,R)$ (as before). Condition (2) will allow us to get unconditional uniqueness via bootstrapping. From the embedding in Lemma \ref{embedding} we have that the solution we constructed lies in $$B_{X^{k,\gamma,T}}(0,2R) \subseteq B_{C([-T,T],V^{\omega,\kappa})}(0,2R) \subseteq B_{X^{\kappa - \varepsilon, \gamma,T}}(0,2RQ),$$ and is thus the unique fixed point of $\Gamma_{\tilde{u}_0}$ in $B_{X^{\kappa - \varepsilon, \gamma,T}}(0,2RQ)$. At first, this only gives uniqueness amongst functions in $$C([-T,T],V^{\omega,\kappa})\cap B_{X^{\kappa-\varepsilon, \gamma,T}}(0,2RQ).$$ By bootstrapping, we upgrade this to unconditional uniqueness in $C([-T,T],V^{\omega,\kappa})$: 
  Let $u_1, u_2 \in C([-T,T],V^{\omega,\kappa})$ be two solutions. Consider $$\Omega = \{t \in [-T,T]: \forall x \ u_1(t,x) = u_2(t,x)\}.$$
  $\Omega$ is closed by continuity and nonempty as $0 \in \Omega$. Note that $t \mapsto \|u_{1}\|_{\kappa -\varepsilon, \gamma, t}$ is increasing and continuous in $t$. If we had that both $\|u_{1}\|_{\kappa -\varepsilon, \gamma, t}$ and $\|u_{2}\|_{\kappa -\varepsilon, \gamma, t}$ stayed bounded by $2RQ$ for all $t \leq T$, then we could apply uniqueness in $B_{X^{\kappa -\varepsilon, \gamma, t}}(0,2RQ)$ to get $u_1 = u_2$. Otherwise, there exists a minimal $\tilde{t}$ such that $\|u_{1}\|_{\kappa -\varepsilon, \gamma, \tilde{t}}=2RQ$ or $\|u_{2}\|_{\kappa -\varepsilon, \gamma, \tilde{t}}=2RQ$. By continuity, we have that for some $\delta>0$ the norms $\|u_{1}\|_{\kappa - \varepsilon, \gamma, t}$ and $\|u_{2}\|_{\kappa - \varepsilon, \gamma, t}$ stay within $3RQ$ for $t<\tilde{t}+\delta$. By uniqueness in $B_{\kappa-\varepsilon, \gamma, \tilde{t}+\frac{\delta}{2}}(0,3RQ)$ we conclude that $(-\tilde{t}-\frac{\delta}{2},\tilde{t}+\frac{\delta}{2}) \subseteq \Omega$. So $\Omega$ is open. Hence, by connectedness $\Omega = [-T,T]$.

 Thus, the data-to-solution map $B_{V^{\omega,k}}(0,R) \rightarrow C([-T,T],V^{\omega,\kappa})$ is well-defined. Finally, we show that this data-to-solution map is Lipschitz continuous.  Let $\tilde{u},\tilde{\tilde{u}} \in B_{X^{k,\gamma,T}}(0,2R)$ be the unique solutions for initial data $\tilde{u}_0, \tilde{\tilde{u}}_0 \in B_{V^{\omega,k}}(0,R)$. Then $$\tilde{u} - \tilde{\tilde{u}}= \Gamma_{\tilde{u}_0}(\tilde{u}) - \Gamma_{\tilde{\tilde{u}}_0}(\tilde{\tilde{u}})= \Gamma_{0}(\tilde{u}) - \Gamma_{0}(\tilde{\tilde{u}}) + e^{tL}\tilde{\tilde{u}}_0 - e^{tL}\tilde{\tilde{u}}_0.$$ 
 The contraction property of $\Gamma_0$ on $B_{X^{k,\gamma,T}}(0,2R)$ with contracting factor $q$ say, together with the first linear estimate \eqref{firstlinear} allow us to conclude \begin{equation*}
 \|\tilde{u} - \tilde{\tilde{u}}\|_{k,\gamma,T} \leq \frac{1}{1-q}\|\tilde{u}_0 - \tilde{\tilde{u}}_0 \|_{\omega,k}. \qedhere	
 \end{equation*} 
 \end{proof}
 
 Proving Theorem \ref{multipliertheorem} only requires a minor modification of the argument. It suffices to note that the Banach algebra $X^{k,\gamma,T}$ is also invariant under complex conjugation. The multilinear estimate in Proposition~\ref{multilinear} then yields the following estimates
 \begin{equation}\|P(u,\bar{u})\|_{k,\gamma,T} \lesssim_P (1+\|u\|_{k,\gamma,T})^{\deg P}\end{equation}
 \begin{equation}\|P(u,\bar{u})-P(v,\bar{v})\|_{k,\gamma,T} \lesssim_P [\|u\|_{k,\gamma,T}+\|v\|_{k,\gamma,T}]^{(\deg P) -1}\|u-v\|_{k,\gamma,T}.\end{equation}
 Now the proof carries through as before. Note that for \eqref{dNLS} the multiplier operator $L=-i\partial_x^2$ has symbol $m(\xi)=i\xi^2$ satisfying the growth condition \eqref{growthcondition} so that Theorem \ref{multipliertheorem} is applicable. \\

 Next, we would like to address the relationship between our results and those of \cite{damanik2015existence, damanik2021local}. To this end, we discuss the difference between the spaces defined by condition \eqref{decay} and the spaces $V^{\omega,k}$. 
 Observe that $$\eqref{decay} \iff \widehat{u}(n)e^{k|n|} \in l^{\infty}_n(\mathbb{Z}^{\nu}),$$
  suggesting the introduction of the spaces
  $$W^{\omega,k} = \Big\{ u(x)=\sum_{n \in \mathbb{Z}^{\nu}}\widehat{u}(n)e^{i \langle \omega, n \rangle x} \ \Big| \ \| |\widehat{u}(n)|e^{k|n|}\|_{l^{\infty}_n(\mathbb{Z}^{\nu})} < \infty \Big\},$$ using a weighted $l^{\infty}$-norm on the Fourier side. In comparison, the spaces $V^{\omega,k}$ are defined using a weighted $l^{1}$-norm on the Fourier side, with the same weights:
  $$u \in V^{\omega,k} \iff \widehat{u}(n)e^{k|n|} \in l^{1}_n(\mathbb{Z}^{\nu}).$$
  Thus, $$V^{\omega,k} \subseteq W^{\omega,k}.$$  
  Moreover, we have that \begin{equation}\label{l1linftyspaces} W^{\omega,k} \subseteq V^{\omega, k'} \ \text{for} \ 0<k'<k,\end{equation} as $e^{(k'-k)|n|} \in l_n^1(\mathbb{Z}^{\nu})$. The advantage of working with the $l^1$-based space $V^{\omega,k}$ instead of the $l^{\infty}$-based $W^{\omega,k}$ stems from $l^1(\mathbb{Z}^{\nu})$ being a convolution algebra.

  These relations allow us to recover the existence and uniqueness statement from \cite{damanik2015existence, damanik2021local} as a corollary to Theorem \ref{multipliertheoremreal}. Indeed, Theorem \ref{multipliertheorem} applies to the \eqref{g-KdV} equation by taking the operator $L$ to be $L=\partial_x^3$; its symbol $m(\xi)=-i\xi^3$ satisfies the growth condition \eqref{growthcondition} and the symmetry condition \eqref{symmetrycondition}. 
  In view of \eqref{l1linftyspaces}, $\omega$-quasi-periodic initial data satisfying \eqref{decay} belongs to the space $W^{\omega,k} \subseteq V^{\omega,k'}$ for $\frac{k}{2} < k' < k$. Thus, by Theorem \ref{multipliertheoremreal}, there exists a solution in $C([0,T],V^{\omega,\frac{k}{2}}) \subseteq C([0,T],W^{\omega,\frac{k}{2}})$; note that the Fourier coefficients of a  function in this space satisfy the bound \eqref{decreaseddecay} uniformly in time. In order to recover the uniqueness statement, we note that a solution of the form \eqref{seriesexpansion} whose coefficients satisfy the bound $|\widehat{u}(t,n)| \lesssim e^{- \rho|n|}$ uniformly in time $t \in [0,T]$ for some $T>0$ and some $\rho>0$, belongs to the space $C([0,T],W^{\omega,\rho}) \subseteq C([0,T],V^{\omega,\kappa}$) for $\kappa<\rho$. Taking $\kappa < \min(\rho,k)$ in Theorem~\ref{multipliertheoremreal} yields the claim.

\section{Models without derivative nonlinearity}\label{noderivativemodels}

\subsection*{Nonlinear Schr\"odinger equation} 

We study the Cauchy problem for the nonlinear Schr\"odinger equation 
\begin{equation}\tag{NLS}\label{NLS}
	\partial_tu-i\partial_x^2u \pm i |u|^{2p}u = 0, \quad u(x,0)=u_0(x) 
\end{equation}
(here $p>0$ is an integer) via the associated Duhamel's formula 
\begin{equation}\label{DuhamelNLS}
u=e^{it\partial_x^2}u_0 \mp i\int_0^t e^{i(t-\tau)\partial_x^2}[(|u|^{2p}u)(\tau)] d\tau.
\end{equation}
A natural setting to work in is the Wiener algebra of $\omega$-quasi-periodic functions, as it is preserved under the linear flow given by the propagator $e^{it\partial_x^2}$ and the nonlinearity in \eqref{NLS} is locally Lipschitz. The Wiener algebra of $\omega$-quasi-periodic functions is
$$A^{\omega}:=\Big\{u=\sum_{n\in\mathbb{Z}}\widehat{u}(n)e^{i\langle \omega,n\rangle x} \ \Big| \ \|\widehat{u}\|_{l^1(\mathbb{Z}^\nu)}<\infty\Big\}$$
equipped with the norm $\|u\|_{A^{\omega}}:=\|\widehat{u}\|_{l^1(\mathbb{Z}^{\nu})}=\sum_{n \in \mathbb{Z}^\nu} |\widehat{u}(n)|.$
More generally, for an increasing weight function $\lambda:(0,\infty) \rightarrow (0,\infty)$ bounded away from zero and satisfying $\lambda(q+r) \leq \lambda(q) \cdot \lambda(r)$ we can define the Banach algebra $A^{\omega,\lambda}$ of all $\omega$-quasi-periodic functions for which the norm 
$$ \|u\|_{A^{\omega,\lambda}}:=\|\widehat{u}(n)\lambda(|n|)\|_{l^1(\mathbb{Z}^{\nu})}=\sum_{n \in \mathbb{Z}^{\nu}} |\widehat{u}(n)|\lambda(|n|) $$
is finite. $A^{\omega}$ corresponds to the constant function $\lambda \equiv 1$. The space $V^{\omega,k}$ corresponds to the weight function $\lambda(m)=e^{km}$. Another example of an admissible weight function is $\lambda(m)=2\langle m \rangle^2 := 2(1+m^2)$. We denote by $C([0,T],A^{\omega,\lambda})$ the Banach algebra of continuous function $u:[0,T] \rightarrow A^{\omega,\lambda}$ equipped with the norm $\|u\|_{C_tA^{\omega,\lambda}_x}:=\sup_{t\in[0,T]} \|u(t,\cdot)\|_{A^{\omega,\lambda}}$. We say that $u \in C([-T,T],A^{\omega,\lambda})$ is a solution to \eqref{NLS} if and only if equality holds in \eqref{DuhamelNLS} pointwise in space and time. If $|n|^2 \lesssim\lambda(|n|)$, then functions in $A^{\omega,\lambda}$ are twice continuously differentiable in the spatial variable. Under this assumption on $\lambda$, we have that for functions in $C([-T,T],A^{\omega,\lambda})$, being a classical solution to \eqref{NLS}\footnote{We say that $u\in C([-T,T],A^{\omega,\lambda})$ is a classical solution to NLS if and only if it is once continuously differentiable in time and satisfies \eqref{NLS} pointwise.} is equivalent to Duhamel's formula \eqref{DuhamelNLS}. This can be shown analogously to the proof of Proposition \ref{notionofsolution}. 

We use the the right-hand side of \eqref{DuhamelNLS} to define a mapping
$$\Gamma_{u_0}u =e^{it\partial_x^2}u_0 + \int_0^t e^{i(t-\tau)\partial_x^2}[(|u|^{2p}u)(\tau)] d\tau.$$

In the absence of a derivative in the nonlinearity, we get the following estimates for the norm $\|\cdot \|_{C_tA^{\omega,\lambda}_x}$ by using the fact that the propagator $e^{it\partial_x^2}$ acts by isometry on $A^{\omega,\lambda}$ and by utilizing the multilinear estimate from the algebra property:
\begin{equation}
	\|e^{it\partial_x^2}u_0\|_{C_tA_x^{\omega,\lambda}}=\|u_0\|_{A^{\omega,\lambda}}
\end{equation}
\begin{equation}
	\Big\|\int_0^t e^{i(t-\tau)\partial_x^2}[(|u|^{2p}u)(\tau)] d\tau \Big\|_{C_tA_x^{\omega,\lambda}} \leq T\|u\|_{C_tA_x^{\omega,\lambda}}^{2p+1}
\end{equation}
\begin{equation}
\|\Gamma_{u_0}u-\Gamma_{u_0}v\|_{C_tA^{\omega,\lambda}_x} \leq  T\cdot(\|u\|_{C_tA_x^{\omega,\lambda}}+\|v\|_{C_tA_x^{\omega,\lambda}})^{2p}\cdot\|u-v\|_{C_tA_x^{\omega,\lambda}}.	
\end{equation}

Equipped with these estimates, we obtain unconditional local well-posedness of \eqref{NLS} in $A^{\omega,\lambda}$ both forward and backward in time by standard contraction mapping and bootstrapping arguments analogous to the proof presented in Section~\ref{Proof}:

\begin{theorem}\label{NLStheorem}
    NLS is unconditionally locally well-posed in $A^{\omega,\lambda}$.
\end{theorem}
In particular, we get unconditional local well-posedness in $V^{\omega,k} \subseteq A^{\omega}$. In comparison to Theorem \ref{multipliertheoremreal} and \ref{multipliertheorem}, we also have persistence of regularity; for initial data $u_0 \in V^{\omega,k}$ the solution $u(t)$ to \eqref{NLS} stays locally in $V^{\omega,k}$.

  \subsection*{Generalized Benjamin-Bona-Mahony equation}\label{gBBMsection}

  In the paper \cite{damanik2022quasiperiodic}, the generalized Benjamin-Bona-Mahony equation 
  \begin{equation}\tag{g-BBM}\label{gBBM} \partial_t u - \partial_{xxt} u + \partial_x u + \partial_x(u^{p+1}) = 0 \end{equation} was studied by the same method as in \cite{damanik2015existence,damanik2021local} --- via a combinatorial analysis of Picard iterates for the associated integral equation for the Fourier coefficients. Here, we provide a shorter proof of the existence and uniqueness result in \cite{damanik2022quasiperiodic}; as a biproduct of our arguments, we obtain a statement about continuous dependence of solutions on initial data.
  
  The \eqref{gBBM} equation can be considered as an ODE for a Banach space valued function:
    \begin{equation}\label{gBBMODE}
  \partial_t u = \frac{-\partial_x}{1-\partial_x^2}(u+u^
  {p+1}).	
  \end{equation}
  Here $\frac{-\partial_x}{1-\partial_x^2}$ is the Fourier multiplier operator with symbol $\frac{-i\xi}{1+\xi^2}$. The symbol stays bounded in absolute value, so that $\frac{-\partial_x}{1-\partial_x^2}$ defines a bounded linear operator on $V^{\omega,k}_{\mathbb{R}}$ with norm $$\Big\|\frac{-\partial_x}{1-\partial_x^2}\Big\|_{\mathcal{B}(V_{\mathbb{R}}^{\omega,k})}=\sup_{n \in \mathbb{Z}^{\nu}} \Big| \frac{-i\langle \omega,n\rangle}{1+\langle \omega, n \rangle^2} \Big| \leq \frac{1}{2}.$$
   The treatment of the nonlinearity $\frac{-\partial_x}{1-\partial_x^2}u^{p+1}$ does not involve the complexities associated with the presence of an actual derivative --- an actual derivative would manifest on the Fourier side as linear growth of the multiplier symbol. Boundedness of $\frac{-\partial_x}{1-\partial_{x}^2}$ combined with the Banach algebra property gives that $$F:V^{\omega,k}_{\mathbb{R}} \rightarrow V^{\omega,k}_{\mathbb{R}}, u \mapsto \frac{-\partial_x}{1-\partial_x^2}(u+u^{p+1})$$
  is locally Lipschitz continuous via $$\|F(u)-F(v)\|_{\omega,k} \leq \frac{1}{2}[1+(\|u\|_{\omega,k}+\|v\|_{\omega,k})^p]\|u-v\|_{\omega,k}.$$
  By the Picard-Lindel\"of Theorem, we obtain the unconditional local well-posedness of \eqref{gBBM} in $V^{\omega,k}_{\mathbb{R}}$ both forward and backward in time.
  
\printbibliography 
\end{document}